\documentclass[12pt]{amsart}

\usepackage{amssymb}
\usepackage{fullpage}

\numberwithin{equation}{section}

\theoremstyle{plain}
\newtheorem{theorem}{Theorem}[section]
\newtheorem{lemma}[theorem]{Lemma}
\newtheorem{proposition}[theorem]{Proposition}
\newtheorem{corollary}[theorem]{Corollary}

\theoremstyle{remark}
\newtheorem{example}[theorem]{Example}
\newtheorem{remark}[theorem]{Remark}

\def\mod{\ \mathrm{mod}\ }
\def\exp{\mathrm{exp}}
\def\Dic{\mathrm{Dic}}
\def\Dih{\mathrm{Dih}}
\def\Aut{\mathrm{Aut}}
\def\aut#1{\Aut(#1)}
\def\Hol{\mathrm{Hol}}
\def\hol#1{\Hol(#1)}
\def\Inn{\mathrm{Inn}}
\def\inn#1{\Inn(#1)}
\def\Mlt{\mathrm{Mlt}}
\def\mlt#1{\Mlt(#1)}
\def\Par{\mathrm{Par}}
\def\Iso{\mathrm{Iso}}
\def\Inv{\mathrm{Inv}}

\def\a{\alpha}
\def\Z{\mathbb Z}

\begin{document}

\title{Automorphisms of dihedral-like automorphic loops}

\author[Aboras]{Mouna Aboras}

\author[Vojt\v{e}chovsk\'y]{Petr Vojt\v{e}chovsk\'y}

\address{Department of Mathematics, University of Denver, 2280 S Vine St, Denver, Colorado 80208, U.S.A.}

\email[Aboras]{mmrm9804@yahoo.com}

\email[Vojt\v{e}chovsk\'y]{petr@math.du.edu}

\thanks{Research partially supported by the Simons Foundation Collaboration Grant 210176 to  Petr Vojt\v{e}chovsk\'y.}



\begin{abstract}
Automorphic loops are loops in which all inner mappings are automorphisms. A large class of automorphic loops is obtained as follows: Let $m$ be a positive even integer, $G$ an abelian group, and $\a$ an automorphism of $G$ that satisfies $\alpha^2=1$ if $m>2$. Then the \emph{dihedral-like automorphic loop} $\Dih(m,G,\a)$ is defined on $\Z_m\times G$ by
\begin{displaymath}
    (i,u)(j,v)=(i+j, ((-1)^{j}u+v)\a^{ij}).
\end{displaymath}
We prove that two finite dihedral-like automorphic loops $\Dih(m,G,\a)$, $\Dih(\overline{m},\overline{G},\overline{\a})$ are isomorphic if and only if $m=\overline{m}$, $G=\overline{G}$, and $\a$ is conjugate to $\overline{\a}$ in the automorphism group of $G$. Moreover, for a finite dihedral-like automorphic loop $Q$ we describe the structure of the automorphism group of $Q$ and its subgroup consisting of inner mappings of $Q$.
\end{abstract}

\maketitle

\section{Introduction}

A groupoid $(Q,\cdot)$ is a \emph{quasigroup} if for every $x\in Q$ the translations
\begin{displaymath}
    L_x:Q \to Q,\ yL_x=xy, \quad R_x:Q\to Q,\ yR_{x}=yx
\end{displaymath}
are bijections of $Q$. If $Q$ is a quasigroup with $1\in Q$ such that $x1=1x=x$ holds for every $x\in Q$, then $Q$ is a \emph{loop}.

Given a loop $Q$, the \emph{multiplication group} of $Q$ is the permutation group
\begin{displaymath}
    \mlt{Q} = \langle L_{x}, R_{x} : x\in Q\rangle,
\end{displaymath}
and the \emph{inner mapping group} of $Q$ is the subgroup
\begin{displaymath}
    \inn{Q} = \{\varphi\in\mlt{Q} : 1\varphi = 1 \}.
\end{displaymath}
It is well known, cf. \cite{Br}, that
\begin{displaymath}
    \inn{Q} = \langle L_{x,y},R_{x,y},T_x : x, y\in Q\rangle,
\end{displaymath}
where
\begin{displaymath}
    L_{x,y}=L_xL_yL_{yx}^{-1},\quad R_{x,y} = R_xR_yR_{xy}^{-1},\quad T_x = R_xL_x^{-1}.
\end{displaymath}

An \emph{automorphic loop} (or \emph{A-loop}) is a loop $Q$ in which every inner mapping is an automorphism, that is, $\inn{Q}\le\aut{Q}$. Note that every group is an automorphic loop, but the converse is certainly not true. The study of automorphic loops began with \cite{BrPa}, and many structural results were obtained in \cite{KiKuPhVo}.

\subsection{Dihedral-like automorphic loops}

Consider the following construction: Let $m>1$ be an integer, $G$ an abelian group, and $\a$ an automorphism of $G$. Define $\Dih(m,G,\a)$ on $\Z_m\times G$ by
\begin{equation}\label{Eq:D}
    (i,u)(j,v)=(i+j,((-1)^ju+v)\a^{ij}),
\end{equation}
where we assume that $i$, $j\in\{0,\dots,m-1\}$ and where we do not reduce modulo $m$ in the exponent of $\alpha$. To save space, we will write
\begin{displaymath}
    s_j = (-1)^j.
\end{displaymath}

Among other results, it was proved in \cite{Ab} that $\Dih(m,G,\a)$ is a loop, and that this loop is automorphic if and only if one of the following conditions hold:
\begin{enumerate}
\item[$\bullet$] $m=2$, or
\item[$\bullet$] $m>2$ is even and $\a^2=1$, or
\item[$\bullet$] $m$ is odd, $\a=1$ and $\mathrm{exp}(G)\le 2$, in which case $Q=\Z_m\times G$ is an abelian group.
\end{enumerate}
To avoid uninteresting cases, we say in this paper that $Q=\Dih(m,G,\a)$ is a \emph{dihedral-like automorphic loop} if either $m=2$ or ($m>2$ is even and $\a^2=1$).

The dihedral-like automorphic loops with $m=2$ were first discussed in \cite{KiKuPhVo}, particularly the case $\Dih(2,\Z_n,\a)$.

Dihedral-like automorphic loops are of interest because they account for many small automorphic loops. For instance, by \cite[Corollary 9.9]{KiKuPhVo}, an automorphic loop of order $2p$, with $p$ an odd prime, is either the cyclic group $\Z_{2p}$ or a loop $\Dih(2,\Z_p,\alpha)$.

\subsection{Notational conventions}

Throughout the paper we will use extensively the following properties and conventions:
\begin{itemize}
\item the word ``nonassociative'' means ``not associative,''
\item automorphic loops are power-associative \cite{BrPa}, and hence powers and two-sided inverses are well-defined,
\item the subgroup $\langle 2\rangle$ of $\Z_m$ will be denoted by $E$,
\item since $m$ is even, we have $s_{i\mod m} = s_i$ for every integer $i$, and it therefore does not matter whether we reduce modulo $m$ in the subscript of $s_i$,
\item $s_{j+k} = s_js_k$ for every $j$, $k$,
\item $(s_ju)\a = s_j(u\a)$ for every $j$ and $u$,
\item except for the exponents of $\a$ and other automorphisms, all statements should be read modulo $m$. (For instance, $E=0$ if $m=2$.)
\end{itemize}
If $\a^2=1$ then $\a^i = \a^{i\mod m}$ for every integer $i$, so it does not matter in this case whether we reduce modulo $m$ in the exponent of $\a$. If $\a^2\ne 1$ (and thus $m=2$) we will assume that all variables are taken from $\{0,1\}$, and we employ $i\oplus j$ whenever necessary to distinguish the addition in $\Z_m$ from the addition of integers.

Consequently, we have
\begin{displaymath}
    \a^i\a^j = \a^{i+j}
\end{displaymath}
in all situations, a key property in calculations. Since $ij+(i\oplus j)k = i(j\oplus k)+jk$ holds for every $i$, $j$, $k\in\{0,1\}$ when $m=2$, we also have
\begin{equation}\label{Eq:Exp}
    \a^{ij+(i\oplus j)k} = \a^{i(j\oplus k)+jk}
\end{equation}
in all situations.

\subsection{Groups among dihedral-like automorphic loops}

Groups are easy to spot among dihedral-like automorphic loops:

\begin{lemma}\label{Lm:G}
Let $Q=\Dih(m,G,\a)$ be a dihedral-like automorphic loop. Then $Q$ is a group if and only if $\a=1$, and it is a commutative group if and only if $\a=1$ and $\exp(G)\le 2$.

Moreover, the group $\Dih(m,G,1)$ is a semidirect product $\Z_m\ltimes_\varphi G$ with multiplication
\begin{displaymath}
    (i,u)(j,v) = (i+j,u\varphi_j+v),
\end{displaymath}
where $\varphi:\Z_m\to\aut{G}$, $j\mapsto \varphi_j$ is given by $u\varphi_j = s_ju$.
\end{lemma}
\begin{proof}
We have $(i,u)(j,v)\cdot (k,w) = (i,u)\cdot (j,v)(k,w)$ if and only if
\begin{displaymath}
    (s_k(s_ju+v)\a^{ij}+w)\a^{(i\oplus j)k} = (s_{j+k}u+(s_kv+w)\a^{jk})\a^{i(j\oplus k)},
\end{displaymath}
which holds, by \eqref{Eq:Exp}, if and only if
\begin{equation}\label{Eq:Assoc}
    s_{j+k}u\a^{ij+(i\oplus j)k} + w\a^{(i\oplus j)k} = s_{j+k}u\a^{i(j\oplus k)} + w\a^{jk+i(j\oplus k)}.
\end{equation}
With $u=0$, $k=0$ and $i=j=1$, this reduces to $w=w\a$, so $\a=1$ is necessary. Conversely, if $\a=1$, then \eqref{Eq:Assoc} reduces to the trivial identity $s_{j+k}u+w=s_{j+k}u+w$.

The rest is clear from \eqref{Eq:D}. Note that the mapping $\varphi$ is a homomorphism thanks to $s_{j+k} = s_js_k$.
\end{proof}

We will call associative dihedral-like automorphic loops \emph{dihedral-like groups}. The dihedral-like groups encompass the \emph{dihedral groups} $\Dih(2,\Z_n,1) = D_{2n}$, the \emph{generalized dihedral groups} $\Dih(2,G,1) = \Dih(G)$, and certain generalized dicyclic groups $\Dih(4,G,1)$.

Recall that for an abelian group $A$ and an element $y\in A$ of order two the \emph{generalized dicyclic group} $\Dic(A,y)$ is the group generated by $A$ and another element $x$ such that $x^2=y$ and $x^{-1}ax=a^{-1}$ for every $a\in A$, cf. \cite[p. 170]{MiBlDi}. If $A=\Z_{2n}$ and $y$ is the unique element of order two in $A$, then $\Dic(A,y)=\Dic_{4n}$ is the \emph{dicyclic group}.

It is easy to see that $\Dih(4,G,1)$ is isomorphic to $\Dic(\Z_2\times G,(1,0))$, by letting $A=E\times G$, $y=(2,0)$ and $x=(1,0)$. In particular, if $n$ is odd, then $\Dih(4,\Z_n,1)$ is isomorphic to $\Dic_{4n}$.

Dihedral-like groups contain additional classes of groups. For instance, $\Dih(6,\Z_5,1)$ of order $30$ is obviously not generalized dicyclic, and it is not isomorphic to the unique generalized dihedral group $\Dih(2,\Z_{15},1)$ of order $30$.

On the other hand, not every (generalized) dicyclic group is found among dihedral-like groups, e.g. $\Dic_{16}$. This can be verified directly with the \texttt{LOOPS} \cite{LOOPS} package for \texttt{GAP} \cite{GAP}.

\subsection{Summary of results}

In this paper we prove that two finite dihedral-like automorphic loops $\Dih(m,G,\a)$, $\Dih(\overline{m},\overline{G},\overline{\a})$ are isomorphic if and only if $m=\overline{m}$, $G=\overline{G}$, and $\a$ is conjugate to $\overline{\a}$ in $\aut{G}$. We describe the automorphism groups and inner mapping groups of all finite dihedral-like automorphic loops. We do not know how to generalize these results to infinite dihedral-like automorphic loops.

Note that automorphism groups of generalized dihedral groups are well understood, cf. \cite[p.\ 169]{MiBlDi}. If $\exp(G)\le 2$, then $\Dih(2,G,1)$ is an elementary abelian $2$-group whose automorphism group is the general linear group $GL(2,|G|)$. If $\exp(G)>2$, then $\Aut(\Dih(2,G,1))$ is the \emph{holomorph}
\begin{displaymath}
    \Hol(G) = \Aut(G)\ltimes G,\quad (\a,u)(\beta,v) = (\a\beta,u\beta+v).
\end{displaymath}
We will recover these results as special cases.

\subsection{Related constructions}

Several recent papers deal with constructions of automorphic loops. Recall that for a loop $Q$, the \emph{middle nucleus} is defined as
\begin{displaymath}
    N_\mu(Q) = \{y\in Q : (xy)z=x(yz)\text{ for every }x,\,z\in Q\}.
\end{displaymath}

Following \cite{JeKiVo}, for an abelian group $G$ and a bijection $\a$ of $G$, let $G(\a)$ be defined on $\Z_2\times G$ by $(0,u)(0,v)=(0,u+v)$, $(0,u)(1,v)=(1,u+v)=(1,u)(0,v)$, $(1,u)(1,v)=(0,(u+v)\a)$. By \cite[Corollary 2.3]{JeKiVo}, every commutative loop with middle nucleus of index at most two is of the form $G(\a)$.

We shall see (cf. Lemma \ref{Lm:MiddleNuc}) that all dihedral-like automorphic loops have middle nucleus of index at most two. In \cite[Proposition 2.7]{JeKiVo}, all commutative automorphic loops with middle nucleus of index at most two are described.

Let $G$ be an elementary abelian $2$-group and $\a\in\aut{G}$. Then $G(\a)=\Dih(2,G,\a)$. A special case of \cite[Corollary 2.6]{JeKiVo} then says that for $\a$, $\beta\in\Aut(G)$ with $\a\ne 1\ne \beta$, the loops $\Dih(2,G,\a)$, $\Dih(2,G,\beta)$ are isomorphic if and only if  $\a$, $\beta$ are conjugate in $\aut{G}$.

Examples of centerless commutative automorphic loops can be found in \cite{Na}. Commutative automorphic loops of order $p^3$ are classified in \cite{BaGrVo}. Nuclear semidirect products that yield commutative automorphic loops are studied in \cite{HoJe}. Nonassociative commutative automorphic loops of order $pq$ are constructed in \cite{Dr}.

Finally, \cite{KiKuPhVo} contains additional constructions of (not necessarily commutative) automorphic loops, as well as an extensive list of references on automorphic loops.

\section{Squaring and conjugation}

The squaring map $x\mapsto x^2$ and the conjugation maps $T_x$ are key to understanding the loops $Q = \Dih(m,G,\a)$. For $(i,u)\in Q$, let
\begin{displaymath}
    (i,u)\chi = |\{(j,v)\in Q : (j,v)^2=(i,u)\}|,
\end{displaymath}
that is, $(i,u)\chi$ counts the number of times $(i,u)$ is a square in $Q$.

Denote by $G_2$ the subgroup of $G$ consisting of all elements of $G$ of order at most two.

\begin{lemma}\label{Lm:Squaring}
Let $\Dih(m,G,\a)$ be a finite dihedral-like automorphic loop. Then:
\begin{enumerate}
\item[(i)] $(i,u)\chi\le 2|G|$ for every $i$, $u$,
\item[(ii)] $(i,u)\chi = 0$ for every odd $i$ and every $u$,
\item[(iii)] $(i,u)\chi \le |G|$ whenever $u\ne 0$,
\item[(iv)] $(i,0)\chi = |G|+|G_2|$ when $i$ is even and $m/2$ is odd,
\item[(v)] $(2,0)\chi = 2|G|$ when $m/2$ is even.
\end{enumerate}
\end{lemma}
\begin{proof}
Fix $(i,u)\in Q$. Note that $(j,v)(j,v)=(i,u)$ if and only if
\begin{equation}\label{Eq:Sq1}
    2j\equiv i\pmod m
\end{equation}
and
\begin{equation}\label{Eq:Sq2}
    (s_jv+v)\a^{jj}=u.
\end{equation}
Since $\Z_m\to\Z_m$, $k\mapsto 2k$ is a homomorphism with kernel $\{0,m/2\}$, the congruence \eqref{Eq:Sq1} has either zero solutions or two solutions in $\Z_m$, proving (i).

If $i$ is odd then \eqref{Eq:Sq1} never holds, proving (ii). For the rest of the proof we can assume that $i$ is even, and we denote by $\ell$, $\ell+m/2$ the two solutions to \eqref{Eq:Sq1}.

Suppose that $u\ne 0$. If $\ell$ is odd, then $(s_\ell v+v)\a^{\ell\ell}=0\ne u$ for every $v\in G$, so $(i,u)\chi\le |G|$. Similarly if $\ell+m/2$ is odd, so we can assume that both $\ell$, $\ell+m/2$ are even. Then \eqref{Eq:Sq2} becomes $2v=u$. The mapping $G\to G$, $w \mapsto 2w$  is a homomorphism with kernel $G_2$. If $G=G_2$ then $2v=0$ for all $v\in G$, so $(i,u)\chi=0$. Otherwise $|G_2|\le |G|/2$ and $(i,u)\chi\le 2|G_2|\le |G|$. We have proved (iii) and can assume for the rest of the proof that $u=0$.

Suppose that $m/2$ is odd. Then precisely one of $\ell$, $\ell+m/2$ is odd. Without loss of generality, assume that $\ell$ is odd. With $j=\ell$, \eqref{Eq:Sq2} holds for every $v\in G$. With $j=\ell+m/2$, \eqref{Eq:Sq2} becomes $2v=0$, which holds if and only if $v\in G_2$, proving (iv).

Finally, suppose that $m/2$ is even. Then $\ell=1$ and $\ell+m/2$ are both odd, so $(2,0)\chi = 2|G|$.
\end{proof}

\begin{lemma}\label{Lm:Conjugation}
Let $Q=\Dih(m,G,1)$ be a dihedral-like group. Then
\begin{displaymath}
    (i,u)T_{(j,v)} = (j,v)^{-1}\cdot (i,u)(j,v) = (i, (1-s_i)v+s_ju)
\end{displaymath}
for every $(i,u)$, $(j,v)\in Q$.

In particular, with $I_{(i,u)} = \{(i,u),(i,u)^{-1}\}$, we have:
\begin{enumerate}
\item[(i)] $(0,u)T_{(j,v)}\in I_{(0,u)}$ for every $u\in G$,
\item[(ii)] if $\exp(G)>2$ then for every odd $i$ and every $u\in G$ there is $(j,v)\in Q$ such that $(i,u)T_{(j,v)}\not\in I_{(i,u)}$.
\end{enumerate}
\end{lemma}
\begin{proof}
Note that $(i,u)^{-1} = (-i,-s_iu)$. Since $Q$ is a group and $s_{i+j}=s_is_j$, we have
\begin{align*}
    (i,u)T_{(j,v)} &= (i,u)R_{(j,v)}L_{(j,v)}^{-1} = (j,v)^{-1}\cdot (i,u)(j,v)=\\
        &=(-j,-s_jv)\cdot (i+j,s_ju+v) = (i, s_{i+j}(-s_jv)+s_ju+v) = (i,(1-s_i)v+s_ju),
\end{align*}
as claimed. For $i=0$, we get $(0,u)T_{(j,v)} = (0,s_ju)\in I_{(0,u)}$.

Suppose that $i$ is odd, $\exp(G)>2$ and let $v\in G$ be of order bigger than $2$. We get $(i,u)T_{(0,v)} = (i,2v+u)$, which is different from both $(i,u)$ and $(-i,u)=(-i,-s_iu)=(i,u)^{-1}$.
\end{proof}

\section{The isomorphism problem}\label{Sc:Iso}

Suppose that $\varphi:\Dih(m,G,\a)\to\Dih(\overline{m},\overline{G},\overline{a})$ is an isomorphism of finite dihedral-like loops. Let us first show that $m=\overline{m}$ and $G=\overline{G}$. We will need the following result, which is a combination of \cite[Proposition 9.1]{KiKuPhVo} and \cite[Lemma 4.1]{Ab}. We give a short proof covering both cases. Denote by $E$ the subgroup of $\mathbb Z_m$ generated by $2$.

\begin{lemma}\label{Lm:MiddleNuc}
Let $Q=\Dih(m,G,\a)$ be a nonassociative dihedral-like automorphic loop. Then $N_\mu(Q) = E \times G$.
\end{lemma}
\begin{proof}
Note that $(j,v)\in N_\mu(Q)$ if and only if \eqref{Eq:Assoc} holds for all $(i,u)$, $(k,w)\in Q$. Suppose that $(j,v)\in N_\mu(Q)$. With $u=0$, $i=1$ and $k=0$, \eqref{Eq:Assoc} becomes $w=w\a^j$, so $\a^j=1$. Since $\a\ne 1$, by Lemma \ref{Lm:G}, we conclude that $j\in E$. Conversely, if $j\in E$, then \eqref{Eq:Assoc} reduces to the trivial identity $s_ku\a^{ik} + w\a^{ik} = s_ku\a^{ik} + w\a^{ik}$, and $(j,v)\in N_\mu(Q)$ follows.
\end{proof}

\begin{proposition}\label{Pr:mG}
Let $Q=\Dih(m,G,\a)$ be a finite dihedral-like automorphic loop. Then the parameters $m$, $G$ can be recovered from $Q$.
\end{proposition}
\begin{proof}
Let $s=\max\{(x)\chi : x\in Q\}$ and $S=\{x\in Q;\; (x)\chi = s\}$. By Lemma \ref{Lm:Squaring}, if $m/2$ is odd then $s=|G|+|G_2|$ and $S=E\times 0$, while if $m/2$ is even then $s=2|G|$ and $(2,0)\in S\subseteq E\times 0$. We therefore recover $E\times 0 = \langle S\rangle$ from the known set $S$. Since $|E\times 0|=m/2$, the parameter $m$ is also uniquely determined.

Suppose that $Q$ is a commutative group. Then $\exp(G)\le 2$ by Lemma \ref{Lm:G}. Since $|G|=|Q|/m$ is known, $G$ is uniquely determined.

Now suppose that $Q$ is a group that is not commutative. Then $\exp(G)>2$ by Lemma \ref{Lm:G}. Let $I = \{x\in Q : xT_y\in I_x$ for every $y\in Q\}$. By Lemma \ref{Lm:Conjugation}, $I$ contains $0\times G$ and has empty intersection with $(E+1)\times G$. Hence $\langle S\cup I\rangle = E\times G\le Q$, so $E\times G$ is determined. Since $E$ is known, the Fundamental Theorem of Finite Abelian Groups implies that $G$ is also known.

Finally, suppose that $Q$ is nonassociative. Then $N_\mu(Q)=E\times G$ by Lemma \ref{Lm:MiddleNuc}, and we determine $G$ as above.
\end{proof}

To resolve the isomorphism problem, Propositon \ref{Pr:mG} implies that it remains to study the situation when $\varphi:\Dih(m,G,\a)\to\Dih(m,G,\beta)$ is an isomorphism of finite dihedral-like automorphic loops. By Lemma \ref{Lm:Squaring}, we have
\begin{equation}\label{Eq:Inv1}
    (2,0)\varphi\in E\times 0.
\end{equation}
If $\a\ne 1$ or $\exp(G)>2$, then Lemmas \ref{Lm:Conjugation} and \ref{Lm:MiddleNuc} imply that
\begin{equation}\label{Eq:Inv2}
    (0\times G)\varphi \subseteq E\times G.
\end{equation}
Note that \eqref{Eq:Inv2} can be violated when $\a=1$ and $\exp(G)=2$, for instance by some automorphisms of the generalized dihedral group $\Dih(2,\Z_2\times \Z_2,1)$.

\begin{proposition}\label{Pr:IsoProp}
Suppose that $\varphi:\Dih(m,G,\a)\to \Dih(m,G,\beta)$ is an isomorphism of finite dihedral-like automorphic loops such that either $\a\ne 1$ or $\exp(G)>2$. Then there are $\gamma\in\aut{G}$ and $z\in G$ such that
\begin{enumerate}
    \item[(i)] $(E\times u)\varphi=E\times u\gamma$ for every $u \in G$,
    \item[(ii)] $((E+1)\times u)\varphi=(E+1)\times (z + u\gamma)$ for every $u \in G$,
    \item[(iii)]$\alpha^{\gamma}=\beta$.
\end{enumerate}
\end{proposition}
\begin{proof}
Denote the multiplication in $\Dih(m,G,\beta)$ by $*$, and fix $u\in G$. By \eqref{Eq:Inv2}, $(0,u)\varphi\in E\times v$ for some $v\in G$. We claim that $(E\times u)\varphi = E\times v$. If $(2i,u)\varphi\in E\times v$ for some $i$, then $(2i+2, u)\varphi = ((2,0)(2i,u))\varphi = (2,0)\varphi * (2i,u)\varphi\in (E\times 0)*(E\times v)\subseteq E\times v$, where we have used \eqref{Eq:Inv1}, and where the last inclusion follows from \eqref{Eq:D}. Hence $(E\times u)\varphi\subseteq E\times v$, and the equality holds because $E$ is finite and $\varphi$ is one-to-one.

We can therefore define $\gamma:G\to G$ by $(E\times u)\varphi = E\times u\gamma$. Since $\varphi$ is one-to-one, $\gamma$ is one-to-one. Due to finiteness of $G$, $\gamma$ is also onto $G$. Moreover, $((0,u)(0,v))\varphi = (0,u+v)\varphi\in E\times (u+v)\gamma$ and $(0,u)\varphi * (0,v)\varphi\in (E\times u\gamma) * (E\times v\gamma)\subseteq E \times (u\gamma+v\gamma)$ show that $\gamma$ is a homomorphism, proving (i).

We have seen that $(E\times G)\varphi = E\times G$, and therefore also $((E+1)\times G)\varphi = (E+1)\times G$. Let $z\in G$ be such that $(1,0)\varphi \in (E+1)\times z$. Since $(E+1)\times u = (1,0)(E\times u)$, we have $((E+1)\times u)\varphi = (1,0)\varphi * (E\times u)\varphi \in ((E+1)\times z)*(E\times u\gamma) \subseteq (E+1)\times (z+u\gamma)$, proving (ii).

Finally, we have $((1,0)(1,u))\varphi = (2,u\alpha)\varphi \in E\times u\alpha\gamma$ and $(1,0)\varphi * (1,u)\varphi\in ((E+1)\times z)*((E+1)\times (z+u\gamma))\subseteq E\times (-z+z+u\gamma)\beta=E\times u\gamma\beta$ for every $u\in G$. Hence $\alpha\gamma=\gamma\beta$, proving (iii).
\end{proof}

\begin{theorem}\label{Th:IsoProblem}
Two finite dihedral-like automorphic loops $\Dih(m,G,\a)$ and $\Dih(\overline{m},\overline{G},\overline{\a})$ are isomorphic if and only if $m=\overline{m}$, $G=\overline{G}$ and $\a$ is conjugate to $\overline{\a}$ in $\aut{G}$.
\end{theorem}
\begin{proof}
By Proposition \ref{Pr:mG}, we can assume that $m=\overline{m}$ and $G=\overline{G}$. Let $(Q,\cdot)=\Dih(m,G,\a)$ and $(\overline{Q},*) = \Dih(m,G,\overline{\a})$.

Suppose that $\varphi:Q\to\overline{Q}$ is an isomorphism. If $\a=1$, then $Q$ is a group by Lemma \ref{Lm:G}, hence $\overline{Q}$ is a group, hence $\overline{\a}=1$ by Lemma \ref{Lm:G}, and so $\a$, $\overline{\a}$ are trivially conjugate in $\aut{G}$. We can therefore assume that $\a\ne 1\ne \overline{\a}$. Then $Q$, $\overline{Q}$ are nonassociative by Lemma \ref{Lm:G}, so Proposition \ref{Pr:IsoProp} implies $\a^\gamma=\overline{\a}$ for some $\gamma\in \aut{G}$.

Conversely, suppose that $\a^\gamma=\overline{\a}$ for some $\gamma\in \aut{G}$. Define a bijection $\varphi:Q\to \overline{Q}$ by $(i,u)\varphi=(i,u\gamma)$. Then $((i,u)(j,v))\varphi=(i+j, (s_ju+v)\alpha^{ij})\varphi=(i+j, (s_ju+v)\alpha^{ij}\gamma)$, while $(i,u)\varphi* (j,v)\varphi=(i,u\gamma)*(j,v\gamma)=(i+j,(s_ju\gamma+v\gamma)\overline{\a}^{ij})=(i+j,(s_ju+v)\gamma\overline{\a}^{ij})$. Since $\alpha^{ij}\gamma=\gamma\overline{\a}^{ij}$ holds for every $i$, $j\in \Z_m$ (due to the fact that either $m=2$ or $\a^2=1=\overline{\a}^2$), we see that $\varphi$ is a homomorphism.
\end{proof}

We recover \cite[Corollary 9.4]{KiKuPhVo} as a special case of Theorem \ref{Th:IsoProblem}, using the fact that $\aut{\Z_n}$ is a commutative group:

\begin{corollary}
The dihedral-like automorphic loops $\Dih(2,\Z_n,\a)$, $\Dih(2,\Z_n,\beta)$ are isomorphic if and only if $\a=\beta$.
\end{corollary}

\section{All isomorphisms}

In this section we refine Proposition \ref{Pr:IsoProp} and describe all isomorphisms between finite dihedral-like automorphic loops, except for the case when both $\a=1$ and $\exp(G)\le 2$. In the next section we deal with the special case of automorphisms.

Given two finite dihedral-like automorphic loops $\Dih(m,G,\a)$ and $\Dih(m,G,\beta)$, let
\begin{displaymath}
    \Iso(m,G,\a,\beta)
\end{displaymath}
be the (possibly empty) set of all isomorphisms $\Dih(m,G,\a)\to\Dih(m,G,\beta)$. We will show that there is a one-to-one correspondence between $\Iso(m,G,\a,\beta)$ and the parameter set
\begin{displaymath}
    \Par(m,G,\a,\beta)
\end{displaymath}
consisting of all quadruples
\begin{displaymath}
    (\gamma,z,c,h)
\end{displaymath}
such that:
\begin{itemize}
    \item $\gamma\in\aut{G}$ satisfies $\a^\gamma = \beta$,
    \item $z\in G$,
    \item $c\in\Z_m$ is odd and $\gcd(c,m/2)=1$,
    \item $h$ is a homomorphism $G\to \langle m/2\rangle$ that is trivial if $m/2$ is odd, and that satisfies $\a h=h$ if $m/2$ is even.
\end{itemize}

\begin{remark}\label{Rm:Params}
Let us observe the following facts about the parameters:

Let $\phi$ be the Euler function, that is, $n\phi = |\{k :1\le k\le n$, $\gcd(k,n)=1\}|$. We claim that there are $(m/2)\phi$ choices for $c$ when $m/2$ is odd, and $2\cdot (m/2)\phi$ choices for $c$ when $m/2$ is even. Indeed, there are $(m/2)\phi$ integers $1\le k\le m/2$ such that $\gcd(k,m/2)=1$, and we have $\gcd(k,m/2) = \gcd(k+m/2,m/2)$. If $m/2$ is even, then all such $k$ are necessarily odd, $k+m/2$ is also odd, and so there are $2\cdot (m/2)\phi$ choices for $c$. If $m/2$ is odd, then precisely one of $k$ and $k+m/2$ is odd, and we therefore find precisely $(m/2)\phi$ choices for $c$.

If $h:G\to\langle m/2\rangle\cong\mathbb Z_2$ is a homomorphism, then $K=\ker(h)$ is a subgroup of $G$ of index at most $2$. The condition $\a h = h$ then guarantees that $K\a = K$. Conversely, if $K\le G$ is of index at most two and such that $K\a = K$, then $h:G\to\langle m/2\rangle$ defined by
    \begin{displaymath}
        uh = \left\{\begin{array}{ll}
            0,&\text{ if $u\in K$},\\
            m/2,&\text{ if $u\in G\setminus K$}
        \end{array}\right.
    \end{displaymath}
    is a homomorphism satisfying $\alpha h = h$. We can therefore count the homomorphisms $h$ by counting $\a$-invariant subgroups of index at most $2$ in $G$.
\end{remark}

To facilitate the purported correspondence, define
\begin{displaymath}
    \Psi:\Iso(m,G,\a,\beta)\to\Par(m,G,\a,\beta)
\end{displaymath}
by $\varphi\Psi = (\gamma,z,c,h)$, where
\begin{displaymath}
    (0,u)\varphi = (uh,u\gamma),\quad (1,0)\varphi = (c,z),
\end{displaymath}
and, conversely,
\begin{displaymath}
    \Phi:\Par(m,G,\a,\beta)\to\Iso(m,G,\a,\beta)
\end{displaymath}
by $(\gamma,z,c,h)\Phi = \varphi$, where
\begin{equation}\label{Eq:FromParToIso}
    (i,u)\varphi = (ic+uh, (i\mod 2)z + u\gamma).
\end{equation}

\begin{proposition}\label{Pr:AllIso}
Let $\Dih(m,G,\a)$, $\Dih(m,G,\beta)$ be finite dihedral-like automorphic loops such that either $\a\ne 1$ or $\exp(G)>2$. Then $\Psi:\Iso(m,G,\a,\beta)\to \Par(m,G,\a,\beta)$ and $\Phi:\Par(m,G,\a,\beta)\to \Iso(m,G,\a,\beta)$ are mutually inverse bijections.
\end{proposition}
\begin{proof}
Throughout the proof we write $\bar{i}$ instead of $i\mod 2$. Let $(Q_\a,\cdot) = \Dih(m,G,\a)$, $(Q_\beta,*) = \Dih(m,G,\beta)$, and suppose that $\varphi:Q_\a\to Q_\beta$ is an isomorphism. By Proposition \ref{Pr:IsoProp}, there are $\gamma\in\aut{G}$, $z\in G$, $c\in E+1$ and $h:G\to E$ such that $(0,u)\varphi = (uh,u\gamma)$ and $(1,0)\varphi = (c,z)$. Now, if for some $i\in\Z_m$ we have $(i,0)\varphi = (ic,\bar i z)$, then
\begin{displaymath}
    (i+1,0)\varphi = (i,0)\varphi*(1,0)\varphi = (ic,\bar i z)*(c,z) = ((i+1)c, (-\bar i z+z)\beta^{icc})
\end{displaymath}
because $c$ is odd. When $i$ is odd, the second coordinate becomes $(-z+z)\beta^{icc} = 0 = \overline{i+1}z$, while if $i$ is even, it becomes $z = \overline{i+1}z$. By induction, $(i,0)\varphi = (ic,\bar i z)$ for every $i\in\Z_m$, and we have
\begin{displaymath}
    (i,u)\varphi = (i,0)\varphi*(0,u)\varphi = (ic,\bar i z)*(uh,u\gamma) = (ic+uh,\bar i z+u\gamma),
\end{displaymath}
since $uh$ is even. We have recovered the formula \eqref{Eq:FromParToIso}.

Since $\varphi$ is an isomorphism, we see from $(2i,0)\varphi = (2ic,0)$ that $\{2ic : i\in \Z_m\}=E$, so $\gcd(2c,m)=2$, and $\gcd(c,m/2)=1$. Since $((0, u)(0, v))\varphi=(0, u+v)\varphi=((u+v)h, (u+v)\gamma)$ and $(0, u)\varphi*(0, v)\varphi=(uh, u\gamma)*(vh, v\gamma)=(uh+vh, u\gamma+v\gamma)$, $h$ is a homomorphism $G\to E$. Moreover, $((0,u)(1,0))\varphi = (1,-u)\varphi = (c-uh,z-u\gamma)$ and $(0,u)\varphi*(1,0)\varphi = (uh,u\gamma)*(c,z) = (c+uh,-u\gamma+z)$ show that $2uh=0$ for every $u$, and therefore $h$ is also a homomorphism $G\to\langle m/2\rangle$. If $m/2$ is odd then $E\cap\langle m/2\rangle = 0$, so $h$ is trivial. If $m/2$ is even, we further calculate $((1,0)(1,u))\varphi = (2,u\a)\varphi \in (2c+u\a h)\times G$ and $(1,0)\varphi*(1,u)\varphi = (c,z)*(c+uh,z+u\gamma) \in (2c+uh)\times G$, so $\a h = h$. This means that $(\gamma,z,h,c)\in\Par(m,G,\a,\beta)$, and we have proved along the way that $\Psi\Phi = 1$.

Conversely, let $(\gamma,z,c,h)\in\Par(m,G,\a,\beta)$ be given, and let $\varphi = (\gamma,z,c,h)\Phi$. It is easy to see that $\varphi$ is a bijection, and we proceed to prove that $\varphi$ is a homomorphism. We have
\begin{align*}
    ((i,u)(j,v))\varphi &= (i+j,(s_ju+v)\a^{ij})\varphi\\
        &= ((i+j)c+(s_ju+v)\a^{ij}h,\overline{i+j}z + (s_ju+v)\a^{ij}\gamma),
\end{align*}
and
\begin{align*}
    (i,u)\varphi*(j,v)\varphi &= (ic+uh,\bar i z+u\gamma)*(jc+vh,\bar j z +v\gamma)\\
        &=((i+j)c+(u+v)h, (s_{jc+vh}(\bar i z+u\gamma) + \bar j z+v\gamma)\beta^{(ic+uh)(jc+vh)}).
\end{align*}
Since $\a h = h$ and $h:G\to\langle m/2\rangle$, we see that $(s_ju+v)\a^{ij} h = (s_ju+v)h = (u+v)h$. Since $uh$, $vh\in E$, $c\in E+1$ and either $m=2$ or $\beta^2=1$, we have $\beta^{(ic+uh)(jc+vh)} = \beta^{ij}$. Using $vh\in E$ and $c\in E+1$, we get $s_{jc+vh} = s_j$. It therefore remains to show that
\begin{displaymath}
    \overline{i+j}z + (s_ju+v)\a^{ij}\gamma = (s_j(\bar i z+u\gamma) + \bar j z+v\gamma)\beta^{ij}.
\end{displaymath}
But $\a^\gamma = \beta$, so $\a^{ij}\gamma = \gamma\beta^{ij}$ for all $i$, $j$, and we need to show
\begin{displaymath}
    \overline{i+j}z = (s_j\bar i z + \bar j z)\beta^{ij}.
\end{displaymath}
When $j$ is even, this reduces to the trivial identity $\bar i z = \bar i z$. When $j$ is odd, we need to show
\begin{displaymath}
    \overline{i+1} z = (-\bar i z + z)\beta^i.
\end{displaymath}
When $i$ is even, we get $z=z$. When $i$ is odd, we get $0=(-z+z)\beta^i$.

Hence $\varphi$ is an isomorphism. From \eqref{Eq:FromParToIso} we get $(0,u)\varphi = (uh,u\gamma)$ and $(1,0)\varphi = (c,z)$, proving $\Phi\Psi=1$.
\end{proof}

\begin{example}\label{Ex:Even}
Let $m=12$, $G=\Z_4$, let $\a=\beta$ be the unique nontrivial automorphism of $G$, and let $Q=\Dih(m,G,\a)$. Then $\Iso(m,G,\a,\beta) = \aut{Q}$. There are $2$ choices for $\gamma$ (since $\aut{G}$ is commutative), $4$ choices for $z\in G$, $2\cdot (m/2)\phi = 4$ choices for $c$, and $2$ choices for $h$, corresponding to the subgroups $K=G$ and $K=\{0,2\}$. Altogether, $|\aut{Q}| = |\Par(m,G,\alpha,\beta)| =64$.
\end{example}

\begin{example}\label{Ex:Odd}
Let $m=6$, $u=(1,2)$, $v=(3,4,5,6)$, $G=\langle u,v\rangle\cong \Z_2\times \Z_4$, let $\a\in\aut{G}$ be determined by $u\a = uv^2$, $v\a = v$, and $\beta\in\aut{G}$ by $u\beta = uv^2$, $v\beta = v^3$. Note that $\aut{G}\cong \Dih(2,\Z_4,1)$. Let us calculate $|\Par(m,G,\a,\beta)|$. There are $4$ choices for $\gamma\in\aut{G}$ such that $\a^\gamma=\beta$, $8$ choices for $z\in G$, $(m/2)\phi = 2$ choices for $c$, and $1$ choice for $h$ since $m/2$ is odd. Altogether, $|\Iso(m,G,\a,\beta)| = |\Par(m,G,\a,\beta)| = 64$.
\end{example}

\section{Automorphism groups}

In this section we describe automorphism groups of all finite dihedral-like automorphic loops.

\begin{remark}
The only finite dihedral-like automorphic loops not covered by Proposition \ref{Pr:AllIso} are those loops $Q=\Dih(m,G,\a)$ with $\a=1$ and $G$ a finite abelian group of exponent at most two. A direct inspection of \eqref{Eq:D} shows that then $Q = \Z_m\times G = \Z_m\times \Z_2^t$ for some $t$. Writing $m=r2^s$ with $r$ odd yields $Q = \Z_r\times \Z_{2^s}\times \Z_2^t$. The special case $m=2$ yields $Q=\Z_2^{t+1}$, whose automorphism group is the general linear group over $GF(2)$ and dimension $t+1$, as we have already mentioned. In general, the automorphism group is isomorphic to $\aut{\Z_r}\times \aut{\Z_{2^s}\times \Z_2^t} \cong \Z_r^*\times \aut{\Z_{2^s}\times \Z_2^t}$, and we refer the reader to \cite{HiRh} for more details.
\end{remark}

For an abelian group $G$ and an automorphism $\a$ of $G$ let
\begin{displaymath}
    \Inv_2(G,\a) = \{K\le G : [G:K]\le 2,\,K\alpha = K\}
\end{displaymath}
and let
\begin{displaymath}
    C_{\aut{G}}(\a) = \{\gamma\in\aut{G} : \a\gamma = \gamma\a\}
\end{displaymath}
be the \emph{centralizer} of $\alpha$ in $\aut{G}$.

\begin{proposition}\label{Pr:Size}
Let $\Dih(m,G,\a)$ be a finite dihedral-like automorphic loop such that either $\a\ne 1$ or $\exp(G)>2$. Then
\begin{displaymath}
    |\aut{\Dih(m,G,\a)}| = \left\{\begin{array}{ll}
        |C_{\aut{G}}(\a)|\cdot |G|\cdot (m/2)\phi,&\text{ if $m/2$ is odd},\\
        |C_{\aut{G}}(\a)|\cdot |G|\cdot 2\cdot (m/2)\phi\cdot |\Inv_2(G,\a)|,&\text{ if $m/2$ is even}.
    \end{array}\right.
\end{displaymath}
\end{proposition}
\begin{proof}
This follows from Proposition \ref{Pr:AllIso}, Remark \ref{Rm:Params}, and the fact that $\a^\gamma = \a$ if and only if $\gamma\in C_{\aut{G}}(\a)$.
\end{proof}

Let us write
\begin{displaymath}
    \Par(m,G,\a) = \Par(m,G,\a,\a).
\end{displaymath}
Proposition \ref{Pr:AllIso} allows us to describe the structure of $\aut{\Dih(m,G,\a)}$ by working out the multiplication formula on $\Par(m,G,\a)$.

\begin{theorem}\label{Th:ParMult}
Let $Q = \Dih(m,G,\a)$ be a finite dihedral-like automorphic loop such that either $\a\ne 1$ or $\exp(G)>2$. Then $\aut{Q}$ is isomorphic to $(\Par(m,G,\a),\circ)$, where
\begin{equation}\label{Eq:ParMult}
    (\gamma_0,z_0,c_0,h_0)\circ (\gamma_1,z_1,c_1,h_1) = (\gamma_0\gamma_1, z_0\gamma_1+z_1, c_0c_1+z_0h_1, h_0+\gamma_0h_1).
\end{equation}
\end{theorem}
\begin{proof}
For $0\le i\le 1$ let $\varphi_i = (\gamma_i,z_i,c_i,h_i)\Phi$. Let $\varphi_2 = \varphi_0\varphi_1$, and set $(\gamma_2,z_2,c_2,h_2) = \varphi_2\Psi$. Because $c_0$ is odd and $u{h_0}$ is even, \eqref{Eq:FromParToIso} yields
\begin{gather*}
    (c_2,z_2) = (1,0)\varphi_2 = (1,0)\varphi_0\varphi_1 = (c_0,z_0)\varphi_1 = (c_0c_1+z_0h_1,z_1+z_0\gamma_1),\\
    (uh_2,u\gamma_2) = (0,u)\varphi_2 = (0,u)\varphi_0\varphi_1 = (uh_0,u\gamma_0)\varphi_1 = (uh_0c_1+u\gamma_0h_1,u\gamma_0\gamma_1).
\end{gather*}
The image of $h_0$ is contained in $\langle m/2\rangle$ and $c_1$ is odd, so $h_0c_1 = h_0$. We are done by Proposition \ref{Pr:AllIso}.
\end{proof}

Here are some special cases of interest of Theorem \ref{Th:ParMult}:

\begin{corollary}
Let $Q = \Dih(m,G,\a)$ be a finite dihedral-like automorphic loop such that either $\a\ne 1$ or $\exp(G)>2$.
\begin{enumerate}
\item[(i)] If $m/2$ is odd then $\aut{Q}\cong (C_{\aut{G}}(\a)\ltimes G)\times \Z_{m/2}^*$.
\item[(ii)] If $m=2$ then $\aut{Q}\cong C_{\aut{G}}(\a)\ltimes G\le \hol{G}$.
\item[(iii)] If $\a=1$ then the projection of $\aut{Q} = (\Par(m,G,\a),\circ)$ onto the first two coordinates is isomorphic to $\hol{G}$.
\item[(iv)] If $\a=1$ and $m=2$ (so that $Q$ is a generalized dihedral group) then $\aut{Q}\cong\hol{G}$.
\end{enumerate}
\end{corollary}
\begin{proof}
If $m/2$ is odd, the mappings $h_i$ are trivial. We therefore do not have to keep track of the fourth coordinate in \eqref{Eq:ParMult}, and in the third coordinate we obtain $c_0c_1$. The elements $c_i$ can be identified with automorphisms $g_i$ of $E$ by letting $2g_i = 2c_i$. Parts (i) and (ii) follow.

If $\a=1$ then $C_{\aut{G}}(\a) = \aut{G}$, proving (iii). Part (iv) then follows from (ii) and (iii).
\end{proof}

\section{The inner mapping groups}

We conclude the paper by identifying inner mapping groups as subgroups of the automorphism group for dihedral-like automorphic loops.

\begin{proposition}\label{Pr:InnGens}
Let $Q = \Dih(m,G,\a)$ be a dihedral-like automorphic loop. Then
\begin{align*}
    (k,w)T_{(i,u)} &= (k, (1-s_k)u+s_iw),\\
    (k,w)R_{(j,v),(i,u)} &= (k,w\a^{ij}-s_iu\a^{ij}(\a^{-jk}-1)),\\
    (k,w)L_{(j,v),(i,u)} &= (k,w\a^{ij}+u\a^{ij}(\a^{-jk}-1)).
\end{align*}
for every $(i,u)$, $(j,v)$, $(k,w)\in Q$.
\end{proposition}
\begin{proof}
First note that the three types of generators must preserve the first coordinate. To calculate $T_{(i,u)}$, note that the following conditions are equivalent:
\begin{align*}
    (k,w)T_{(i,u)} &= (k,t),\\
    (k,w)(i,u) &= (i,u)(k,t),\\
    (s_iw+u)\a^{ik} &= (s_ku+t)\a^{ik},\\
    s_iw+u &= s_ku+t,\\
    t &= (1-s_k)u+s_iw.
\end{align*}
For $R_{(j,v),(i,u)}$, the following conditions are equivalent, using \eqref{Eq:Exp}:
\begin{align*}
    (k,w)R_{(j,v),(i,u)} &= (k,t),\\
    (k,w)(j,v)\cdot (i,u) &= (k,t)\cdot (j,v)(i,u),\\
    (s_i(s_jw+v)\a^{jk} +u)\a^{(j\oplus k)i} &= (s_{i+j}t + (s_iv+u)\a^{ij})\a^{(i\oplus j)k},\\
    s_{i+j}w\a^{jk+(j\oplus k)i} + u\a^{(j\oplus k)i} &= s_{i+j}t\a^{(i\oplus j)k} + u\a^{ij+(i\oplus j)k},\\
    t &= w\a^{ij}+s_{i+j}u\a^{ij-jk} - s_{i+j}u\a^{ij}.
\end{align*}
We claim that $s_{i+j}u\a^{ij}(\a^{-jk}-1) = -s_iu\a^{ij}(\a^{-jk}-1)$. Indeed, $s_{i+j}=s_is_j$, so we are done if $j$ is odd, and when $j$ is even then $\a^{-jk}-1=0$ in both sides of the equation.

For $L_{(j,v),(i,u)}$, the following conditions are equivalent, using \eqref{Eq:Exp}:
\begin{align*}
    (k,w)L_{(j,v),(i,u)} &= (k,t),\\
    (i,u)\cdot (j,v)(k,w) &= (i,u)(j,v)\cdot (k,t),\\
    (s_{j+k}u+(s_kv+w)\a^{jk})\a^{i(j\oplus k)} &= (s_k(s_ju+v)\a^{ij}+t)\a^{(i\oplus j)k},\\
    s_{j+k}u\a^{i(j\oplus k)} + w\a^{jk+i(j\oplus k)} & = s_{j+k}u\a^{ij+(i\oplus j)k} + t\a^{(i\oplus j)k},\\
    t &=w\a^{ij} + s_{j+k}u\a^{ij-jk} - s_{j+k}u\a^{ij}.
\end{align*}
Again, we claim that $s_{j+k}u\a^{ij}(\a^{-jk}-1) = u\a^{ij}(\a^{-jk}-1)$. Indeed, if $j$ is even or $k$ is even then $\a^{-jk}-1=0$ in both sides, and if both $j$, $k$ are odd then $s_{j+k}=1$.
\end{proof}

By \cite[Theorem 7.5]{JoKiNaVo}, every automorphic loop satisfies the \emph{antiautomorphic inverse property}, that is, $(xy)^{-1}=y^{-1}x^{-1}$. With $J:x \mapsto x^{-1}$ the inversion map, the antiautomorphic inverse property can be rewritten as $L_xJ =JR_{xJ}$, so it follows that $L_{x,y}J = JR_{xJ,yJ}$. Since $L_{x,y}$ is an automorphism in automorphic loops, we also have $L_{x,y}J = JL_{x,y}$. Altogether, we have obtained:

\begin{lemma}[\cite{KiKuPhVo}]\label{Lm:LR}
In an automorphic loop, $L_{x,y} = R_{x^{-1},y^{-1}}$.
\end{lemma}

Recall that $\inn{Q} = \langle T_x,R_{x,y},L_{x,y} : x,y\in Q\rangle$, and also consider the \emph{right} and \emph{left inner mapping groups}
\begin{displaymath}
    \Inn_{r}(Q) = \langle R_{x,y} : x,y\in Q\rangle,\quad
    \Inn_{\ell}(Q) = \langle L_{x,y} : x,y\in Q\rangle.
\end{displaymath}

\begin{theorem}
Let $Q=\Dih(m,G,\a)$ be a finite dihedral-like automorphic loop such that either $\a\ne 1$ or $\exp(G)>2$. Then:
\begin{enumerate}
\item[(i)] $\Inn_r(Q)=\Inn_{\ell}(Q)$ is isomorphic to the subgroup $\langle\a\rangle\ltimes G(1-\a)$ of $\hol{G}$,
\item[(ii)] $\Inn(Q) = \langle T_x,R_{x,y}:x,y\in Q\rangle = \langle T_x,L_{x,y}:x,y\in Q\rangle$ is isomorphic to the subgroup $(\pm\langle\a\rangle)\ltimes (2G + G(1-\a))$ of $\hol{G}$.
\end{enumerate}
\end{theorem}
\begin{proof}
We obtain $\Inn_r(Q)=\Inn_{\ell}(Q)$ and $\Inn(Q) = \langle T_x,R_{x,y}:x,y\in Q\rangle = \langle T_x,L_{x,y}:x,y\in Q\rangle$ as an immediate consequence of Lemma \ref{Lm:LR}.

Let us now verify that $\langle\a\rangle\ltimes G(1-\a)$ and $(\pm\langle\a\rangle) \ltimes (2G+G(1-\a))$ are subgroups of $\hol{G}$. Indeed, for $\delta$, $\varepsilon\in\{1,-1\}$, we have
\begin{align*}
    (\delta\a^i,2u+v(1-\a))(\varepsilon\a^j,2w+z(1-\a)) &= (\delta\varepsilon\a^{i+j},(2u+v(1-\a))\varepsilon\a^j + 2w + z(1-\a))\\
        &=(\delta\varepsilon\a^{i+j},2(u\varepsilon\a^j+w) + (v\varepsilon\a^j+z)(1-\a)),
\end{align*}
since $1-\a$ commutes with $\pm\a^j$.

By Proposition \ref{Pr:InnGens}, we have
\begin{align*}
    (0,w)T_{(i,u)} = (0,s_iw),\quad& (1,0)T_{(i,u)} = (1,2u),\\
%    (0,w)R_{(j,v),(i,u)} = (0,w\a^{ij}),\quad &(1,0)R_{(j,v),(i,u)} = (1,s_{i+j}u\a^{ij}(\a^{-j}-1)),\\
    (0,w)L_{(j,v),(i,u)} = (0,w\a^{ij}),\quad &(1,0)L_{(j,v),(i,u)} = (1,u\a^{ij}(\a^{-j}-1)).
\end{align*}
Therefore
\begin{align*}
    T_{(i,u)}\Psi &= (s_i,2u,1,0),\\
%    R_{(j,v),(i,u)}\Psi &= (\a^{ij},s_{i+j}u\a^{ij}(\a^{-j}-1),1,0),\\
    L_{(j,v),(i,u)}\Psi &= (\a^{ij},u\a^{ij}(\a^{-j}-1),1,0),
\end{align*}
where we identity the sign $s_i$ with the automorphism $u\mapsto s_iu$. Hence $\inn{Q}$ is nontrivial only on the first two coordinates of $\Par(m,G,\a)$, and it can be identified with a subgroup of $\hol{G}$, by Proposition \ref{Pr:AllIso} and Theorem \ref{Th:ParMult}.

With $u=0$ and $i=j=1$ we get $(\a,0)$ from $L_{(j,v),(i,u)}$. Taking powers of this element, we see that $\langle\a\rangle\times 0\subseteq\Inn_\ell(Q)$. With $i=0$ and $j=1$ we get $(1,u(\a^{-1}-1))=(1,u\a^{-1}(1-\a))$ from $L_{(j,v),(i,u)}$, showing that $1\times G(1-\a)\subseteq\Inn_\ell(Q)$. Therefore $(\langle\a\rangle\times 0)(1\times G(1-\a)) = \langle\a\rangle\times G(1-\a)\subseteq\Inn_\ell(Q)$. On the other hand, using the fact that either $m=2$ or $\a^2=1$, it is easy to see that the generators $L_{(j,v),(i,u)}$ belong to $\langle\a\rangle\times G(1-\a)$. This proves (i).

With $u=0$ we get $(s_i,0)$ from $T_{(i,u)}$, and with $i=0$ we get $(0,2u)$ from $T_{(i,u)}$. Hence $(\pm 1)\times 0$ and $0\times 2G$ are also subsets of $\inn{Q}$, and thus $(\pm\langle\a\rangle)\ltimes (2G + G(1-\a))\subseteq\inn{Q}$. The other inclusion is again clear from an inspection of $T_{(i,u)}$.
\end{proof}

\begin{corollary}
Let $Q=\Dih(2,G,1)$ be a generalized dihedral group such that $\exp(G)>2$. Then the inner automorphism group of $Q$ is isomorphic to $\Z_2\ltimes 2G$.
\end{corollary}

\begin{example}
Let $p$ be an odd prime, $G=\Z_p$, and $\a\ne 1$ the unique involutory automorphism of $G$, that is, $\a=-1$. Let $Q=\Dih(2,G,\a)$. Then $\pm\langle\a\rangle\cong\Z_2\cong\langle\a\rangle$, and $G(1-\a) = 2G = G$. Therefore $\Inn_\ell(Q) = \Inn_r(Q) = \inn{Q}\cong \Z_2\ltimes \Z_p$ is isomorphic to the dihedral group of order $2p$.
\end{example}


\begin{thebibliography}{99}

\bibitem{Ab}
M.~Aboras, \emph{Dihedral-like constructions of automorphic loops}, Comment. Math. Univ. Carolin. \textbf{55}, \textbf{3} (2014), 269--284.
\bibitem{Br}
R.H.~Bruck, \emph{A survey of binary systems}, Ergebnisse der Mathematik und ihrer Grenzgebiete, Springer Verlag, Berlin-G\"ottingen-Heidelberg, 1958.

\bibitem{BrPa}
R.H.~Bruck and L.J.~Paige, \emph{Loops whose inner mappings are automorphisms}, Ann. of Math. (\textbf{2}) \textbf{63} (1956), 308--323.

%\bibitem{Co}
%H.S.M.~Coxeter, \emph{Regular complex polytopes}, second edition, Cambridge University Press, Cambridge, 1991.

\bibitem{BaGrVo}
D.A.S.~De Barros, A.~Grishkov and P.~Vojt\v{e}chovsk\'y, \emph{Commutative automorphic loops of order $p^3$}, J.~Algebra Appl. \textbf{11} (2012), no. \textbf{5}, 1250100.

\bibitem{Dr}
A.~Dr\'apal, \emph{A class of commutative loops with metacyclic inner mapping groups}, Comment. Math. Univ. Carolin. \textbf{49} (2008), no. \textbf{3}, 357--382.

\bibitem{GAP}
The GAP Group, GAP - Groups, Algorithms, and Programming, Version 4.5.7; 2012. \emph{http://www.gap-system.org.}

\bibitem{HiRh}
C.~J.~Hillar and D.~L.~Rhea, \emph{Automorphisms of finite abelian groups}, Amer. Math. Monthly \textbf{114} (2007), no. \textbf{10}, 917--923.

\bibitem{HoJe}
J.~Hora and P.~Jedli\v{c}ka, \emph{Nuclear semidirect product of commutative automorphic
loops}, J. Algebra Appl. \textbf{13} (2014), 1350077.

\bibitem{JeKiVo}
P.~Jedli\v{c}ka, M.~Kinyon and P.~Vojt\v{e}chovsk\'y, \emph{Constructions of commutative automorphic loops}, Comm. Algebra \textbf{38} (2010), no. \textbf{9}, 3243--3267.

\bibitem{JoKiNaVo}
K.W.~Johnson, M.K.~Kinyon, G.P.~Nagy and P.~Vojt\v{e}chovsk\'y, \emph{Searching for small simple automorphic loops}, London Mathematical Society Journal of Computation and Mathematics \textbf{14} (2011), 200--213.

\bibitem{KiKuPhVo}
M.K.~Kinyon, K.~Kunen, J.D.~Phillips, and P.~Vojt\v{e}chovsk\'y, \emph{ The structure of automorphic loops}, to appear in Transactions of the American Mathematical Society.

\bibitem{MiBlDi}
G.A.~Miller, H.~F.~Blichfeldt and L.~E.~Dickson, \emph{Theory and Applications of Finite Groups}, John Wiley \& Sons, Inc., 1916.

\bibitem{Na}
G.P.~Nagy, \emph{On centerless commutative automorphic loops}, Comment. Math. Univ. Carolin. \textbf{55}, \textbf{4} (2014), 485--491.

\bibitem{LOOPS}
G.P.~Nagy and P.~Vojt\v{e}chovsk\'y, LOOPS: Computing with quasigroups and loops, version 2.2.0, a package for GAP, available at \texttt{http://www.math.du.edu/loops}.

\end{thebibliography}
\end{document}